\documentclass[13pt]{article}  % list options between brackets
\usepackage{amssymb}              % list packages between braces
\usepackage{amsthm}
\usepackage{eucal}
\usepackage{verbatim}
\usepackage[dvips]{graphicx}
\usepackage{multirow}
\usepackage{fancyhdr}
\usepackage{enumerate}
\usepackage{fullpage}
\usepackage{amsmath}
\newtheorem{theorem}{Theorem}
\newtheorem{lemma}{Lemma}
\newtheorem{definition}{Definition}
\newtheorem{remark}{Remark}
\newtheorem{proposition}{Proposition}

\makeatletter
\newcommand{\leqnomode}{\tagsleft@true}
\newcommand{\reqnomode}{\tagsleft@false}
\makeatother
\begin{document}

\title{Collision-avoiding in
the singular Cucker-Smale model with
nonlinear velocity couplings.}   % type title between braces
\author{Ioannis Markou}         % type author(s) between braces
\date{March 13 2018}    % type date between braces
\maketitle

\begin{abstract}
Collision avoidance is an interesting feature of the Cucker-Smale
(CS) model of flocking that has been studied in many works, e.g.
\cite{AhChHaLe, AgIlRi, CaChMuPe, CuDo1, CuDo2, MuPe, Pe1, Pe2}. In
particular, in the case of singular interactions between agents, as
is the case of the CS model with communication weights of the type
$\psi(s)=s^{-\alpha}$ for $\alpha \geq 1$, it is important for
showing global well-posedness of the underlying particle dynamics.
In \cite{CaChMuPe}, a proof of the non-collision property for
singular interactions is given in the case of the linear CS model,
i.e. when the velocity coupling between agents $i,j$ is
$v_{j}-v_{i}$. This paper can be seen as an extension of the
analysis in \cite{CaChMuPe}. We show that particles avoid collisions
even when the linear coupling in the CS system has been substituted
with the nonlinear term $\Gamma(\cdot)$ introduced in \cite{HaHaKi}
(typical examples being $\Gamma(v)=v|v|^{2(\gamma -1)}$ for $\gamma
\in (\frac{1}{2},\frac{3}{2})$), and prove that no collisions can
happen in finite time when $\alpha \geq 1$. We also show uniform
estimates for the minimum inter-particle distance, for a
communication weight with expanded singularity
$\psi_{\delta}(s)=(s-\delta)^{-\alpha}$, when $\alpha \geq 2\gamma$,
$\delta \geq 0$.

\end{abstract}


\textbf{2010 MR Subject Classification}: 82C22, 92D50.

\section{Introduction.}

A Cucker-Smale (CS) type of model deals with an interacting system
of $N$ autonomous, self-driven particles (agents). The main model
postulate is that the agents adjust their velocities by taking a
weighted average of their relative velocities to all other agents.
If we let $(x_{i},v_{i})\in \mathbb{R}^d \times \mathbb{R}^d$ be the
phase space position of the i'th particle for $1 \leq i \leq N$, and
$d\geq 1$ be the physical dimension, the dynamics of the particle
motion is governed by the system:

\begin{align} \label{CS} \left\{ \begin{array}{ll} \frac{d}{dt}x_{i}(t)&=v_{i}(t),\quad i=1
\ldots, N, \quad t>0 ,
\\ \\ \frac{d}{dt}v_{i}(t)&=\frac{1}{N}\sum
\limits_{j} \psi(|x_{i}-x_{j}|)(v_{j}-v_{i}) \quad,
\end{array} \right.
\end{align}
given some initial data $(x_{i}(0),v_{i}(0))=(x_{i0},v_{i0})$,
$i=1,\ldots,N$. Throughout this paper, the symbol $\sum \limits_{i}$
is used as an abbreviation of $\sum \limits_{1 \leq i \leq N}$. The
function $\psi(r)$, $r\geq 0$, quantifies the interaction between
two agents and it is to be referred as the communication weight of
the interaction. It is positive, nonincreasing, and vanishes as $r
\to \infty$. We observe that in our model, $\psi(\cdot)$ depends on
the metric distance between two agents. The main question that
arises in the study of \eqref{CS} is whether the system emerges to a
\textit{flock}, i.e., all the particle velocities align
asymptotically in time and the agents stay connected forever. The
prototype example in the CS model was $\psi(r)=(1+r^{2})^{-\beta}$,
for $\beta \geq 0$. In \cite{CuSm1,CuSm2,HaLiu} it was shown that
flocking is guaranteed if $\beta \leq \frac{1}{2}$, and if $\beta >
\frac{1}{2}$ then the system might converge to a flock only under
certain conditions on the initial positions and velocities. The
phase transition that happens when $\beta=\frac{1}{2}$ is typical of
the system \eqref{CS} and supports the more general result that when
weight $\psi(\cdot)$ has a non integrable tail (i.e.
$\int^{\infty}\psi(s)ds=\infty$) then flocking occurs regardless of
the initial configuration of agents.

After its introduction in \cite{CuSm1,CuSm2} (based on an earlier
idea from \cite{ViCzBJco}), research on the CS model took several
different routes. The original flocking results were simplified and
improved in \cite{HaLiu,HaTad}. The CS system was studied in the
presence of Rayleigh friction forces in \cite{HaHaKim}, as well as
other repulsion/alignment/turning forces \cite{AgIlRi}. The effect
of a flock leader in emergent behavior was considered in \cite{Sh}.
The model was also studied with extra random noise terms in
\cite{CuMo,HaLeLe,ToLiYa}. S. Motsch and E. Tadmor proposed a model
that resolves some of the drawbacks of the CS system by normalizing
the communication weights in \cite{MoTad1} and established flocking
conditions. The CS system was studied with delay terms in
\cite{ChHa,ErHaSu}.

An interesting variation to system \eqref{CS} was proposed in
\cite{HaHaKi} and describes the particle system where the linear
coupling term $v_{j}-v_{i}$ is substituted by a nonlinear vector
$\Gamma(v_{j}-v_{i}):\mathbb{R}^d \to \mathbb{R}^d$, i.e.
\begin{align} \label{NL-CS}\left\{ \begin{array}{ll} \frac{d}{dt}x_{i}(t)&=v_{i}(t),\quad i=1
\ldots, N, \quad t>0 ,
\\ \\ \frac{d}{dt}v_{i}(t)&=\frac{1}{N}\sum
\limits_{j} \psi(|x_{i}-x_{j}|)\Gamma(v_{j}-v_{i}) .
\end{array} \right.
\end{align}
The justification given in \cite{HaHaKi} of this nonlinear version
of \eqref{CS} lies in the fact that there seems to be no underlying
physical principle that requires most alignment models to be linear,
other than a modeling convenience. It is therefore of paramount
importance to know that model \eqref{CS} is robust under small
variations in all parameters, including the velocity couplings. We
refer to system \eqref{NL-CS} from now on as NL CS (nonlinear CS).
The continuous coupling vector $\Gamma(v_{j}-v_{i})$ that appears in
\eqref{NL-CS} has the following properties:
\begin{itemize}
\item (A1) (skew symmetry) $\Gamma(-v)=-\Gamma(v)$ for $v \in
\mathbb{R}^d$.
\item (A2) (coercivity) There exists some $C_{1}>0$ and
$\gamma \in (\frac{1}{2},\frac{3}{2})$ such that $\langle \Gamma(v),
v\rangle \geq C_{1} |v|^{2 \gamma}$. Here by $\langle \cdot , \cdot
\rangle$ we denote the inner product in $\mathbb{R}^d$ and with
$|\cdot|$ its induced norm.
\end{itemize}
System \eqref{NL-CS} exhibits similar phase transition properties as
\eqref{CS}. For a communication weight with a non integrable tail
and under assumptions $(A1)-(A2)$, it can be shown that flocking
occurs for $\gamma \in (\frac{1}{2},\frac{3}{2})$ with a rate that
depends explicitly on the value of $\gamma$. In more detail, when
$\gamma \in (\frac{1}{2},1)$ we have flocking that occurs in finite
time $T^{*}<\infty$, algebraically fast. If $\gamma \in
(1,\frac{3}{2})$ we have emergence of a flock in infinite time
$T^{*}=\infty$, with algebraic decay rate. Finally, the case
$\gamma=1$ reduces to the linear case with $T^{*}=\infty$, and
flocking that happens at an exponential decay rate.

A question that requires special investigation is the presence of
collisions between agents. The original CS systems \eqref{CS} and
\eqref{NL-CS} with the weights we just mentioned does not exclude
the possibility of collisions. For obvious reasons, the modeling of
alignment in animal flocks, aerial vehicles, unmanned drones etc.
should in many cases incorporate a mechanism for avoiding
collisions. The way to design systems like that is by introducing an
interaction that becomes singular when two particles collide. This
interaction might have the form of an extra repulsion forcing term,
like in \cite{AhChHaLe, AgIlRi, CuDo1, CuDo2}, or it might simply be
a communication weight that is singular at the origin, e.g.
\cite{CaChMuPe, MuPe, Pe1, Pe2}. In this work our focus shifts to
the latter scenario.

The purpose of this article is to study the problem of collision
avoidance for the NL CS model \eqref{NL-CS}. We prove that for all
the cases where flocking is possible, we have the absence of
collisions in finite time for any interaction of the type
$\psi(s)=s^{-\alpha}$, with $\alpha \geq 1$. This result is in
complete agreement with the linear case treated in \cite{CaChMuPe}.
Our approach shows that the methodology used in \cite{CaChMuPe} is
not just specific to the linear case but can be easily adopted to a
non linear scenario. It also serves as a further indication of the
robustness of the choice of linear couplings in the classical CS
model.

Furthermore, we derive uniform estimates for the general case of a
weight with expanded singularity $\psi(s)=(s-\delta)^{-\alpha}$,
$\delta \geq 0$. We use distance functions of the type
$\mathcal{L}^{\beta}(t)=\frac{1}{N(N-1)}\sum \limits_{i\neq
j}(|x_{i}(t)-x_{j}(t)|-\delta)^{-\beta}$ for some appropriately
chosen $\beta=\beta(\alpha, \gamma) >0$, and prove that
$\mathcal{L}^{\beta}(t)\leq O(T)$, $\forall t \in[0,T]$, when
$\alpha \geq 2 \gamma$. This gives an estimate for the minimum
interparticle distance like $\inf \limits_{i \neq
j}|x_{i}(t)-x_{j}(t)|>O((T N^2)^{-\frac{1}{\beta}})$, which is
enough to conclude the well-posedness of the dynamics for a fixed
number $N$ of agents (given that $\Gamma(\cdot)$ is also Lipshitz).
Unfortunately, this estimate is useless as $N\to \infty$ and leaves
the question of passage to the mean field equation (for singular
communication weights) still open, see e.g. \cite{HaJa,Ja}.

The rest of this paper is structured as follows. In Section 2 we
briefly review the theory for problem \eqref{NL-CS} and the flocking
result presented in \cite{HaHaKi}. In the end of Section 2 we
present the main result of this paper and give its proof in Section
3. Finally, in Section 4 we give and prove the uniform estimates in
the case of a communication weight $\psi(s)=(s-\delta)^{-\alpha}$,
for $\alpha \geq 2\gamma$ and $\delta \geq 0$.

\section{Preliminaries and Main result.}

In what follows, we denote with $x(t)$, $v(t)$ the position and
velocity of the whole $N$ particle system, i.e.
$x(t):=(x_{1},\ldots,x_{N})$ and $v(t):=(v_{1},\ldots,v_{N})$. We
also denote with $(x(t),v(t))$ a solution to the CS system if
$x(t)$, $v(t)$ solve the NL system at time $t$. In the same spirit,
the notation
$(x_{0},v_{0}):=(x_{10},\ldots,x_{N0},v_{10},\ldots,v_{N0})$
represents the vector of initial data. We now give a formal
definition of flocking for a particle system $(x(t),v(t))$.
\begin{definition}[Asymptotic flocking.] A given particle system $(x(t),v(t))$
is said to converge to a flock, iff the following two conditions
hold,
\begin{align} \label{flock} \sup_{t>0} \sup_{i,j}|x_{i}(t)-x_{j}(t)|<\infty , \qquad
\lim \limits_{t\to \infty}\sup_{i,j}|v_{i}(t)-v_{j}(t)| =0
.\end{align}
\end{definition}
We need to keep in mind that the definition we just gave is
independent of the configuration of initial velocities and positions
$(x_{0},v_{0})$. This definition corresponds to the so called
\textit{unconditional} flocking scenario. If flocking holds for a
certain class of initial configurations then we speak of
\textit{conditional} flocking.

Now that we stated the definition of flocking, we may proceed with
the invariants of particle dynamics for system \eqref{NL-CS}. For
this, we define the first three moments of particle motion,
\begin{equation} \label{mom-def} m_{0}(t):=\sum \limits_{i}1,
\quad m_{1}(t)=\sum \limits_{i} v_{i}, \quad m_{2}(t):=\sum
\limits_{i}|v_{i}|^2 .\end{equation} The following lemma shows how
these moments propagate in time.

\begin{lemma}[propagation of moments (see \cite{HaHaKi})] Assume
that the conditions $(A1)$-$(A2)$ hold. Suppose also that
$(x(t),v(t))$ is a solution to the NL CS system. Then, the three
velocity moments satisfy
\begin{equation}\label{mom-eq}
\frac{d}{dt}m_{0}(t)=\frac{d}{dt}m_{1}(t)=0, \qquad
\frac{d}{dt}m_{2}(t)\leq -\frac{C_{1}}{N}\sum
\limits_{i,j}\psi(|x_{i}-x_{j}|)|v_{i}-v_{j}|^{2\gamma} .
\end{equation}
\end{lemma}

\begin{proof}
We give the short proof for completion. The first equation is
trivial since $m_{0}(t)=N$. The equation for $m_{1}(t)$ follows from
the symmetry of $\psi(\cdot)$ and $(A1)$,
\begin{align*}\dot{m}_{1}(t)&=\sum \limits_{i}\dot{v}_{i}(t)=\frac{1}{N}\sum
\limits_{i,j}\psi(|x_{i}-x_{j}|)\Gamma(v_{j}-v_{i})\stackrel{i
\leftrightarrow j}{=} \frac{1}{N}\sum
\limits_{i,j}\psi(|x_{i}-x_{j}|)\Gamma(v_{i}-v_{j})\\
&\stackrel{(A1)}{=}-\frac{1}{N}\sum
\limits_{i,j}\psi(|x_{i}-x_{j}|)\Gamma(v_{j}-v_{i})=0,
\end{align*}
where $\sum \limits_{i,j}=\sum \limits_{i} \sum \limits_{j}$. For
the second moment $m_{2}(t)$ we have
\begin{align*} \dot{m}_{2}(t)&=\frac{d}{dt}\sum \limits_{i}|v_{i}(t)|^2
=2\sum \limits_{i}\langle \dot{v}_{i},v_{i}\rangle =\frac{2}{N}\sum
\limits_{i,j}\psi(|x_{i}-x_{j}|)\langle
\Gamma(v_{j}-v_{i}),v_{i}\rangle \\& \stackrel{i \leftrightarrow
j}{=}\frac{2}{N}\sum \limits_{i,j}\psi(|x_{i}-x_{j}|)\langle
\Gamma(v_{i}-v_{j}),v_{j}\rangle \stackrel{(A1)}{=}- \frac{2}{N}\sum
\limits_{i,j}\psi(|x_{i}-x_{j}|)\langle
\Gamma(v_{j}-v_{i}),v_{j}\rangle\\&=-\frac{1}{N}\sum
\limits_{i,j}\psi(|x_{i}-x_{j}|)\langle
\Gamma(v_{j}-v_{i}),v_{j}-v_{i}\rangle  \stackrel{(A2)}{\leq}
-\frac{C_{1}}{N}\sum
\limits_{i,j}\psi(|x_{i}-x_{j}|)|v_{i}-v_{j}|^{2\gamma}.
\end{align*}
\end{proof}
A direct consequence of the invariance of the first moment is that
the bulk velocity $v_{c}(t):=\frac{1}{N} \sum \limits_{i}v_{i}$
remains constant in time, i.e., $v_{c}(t)=v_{c}(0)$. For the mean
position vector $x_{c}(t):=\frac{1}{N} \sum \limits_{i}x_{i}$, we
easily see that $x_{c}(t)=v_{c}(0)t+x_{c}(0)$. Based on these
observations we can define the standard deviation of positions and
velocities for a group of $N$ agents by
\begin{align*} \sigma_{x}(t):=\sqrt{\frac{1}{N}\sum
\limits_{i}|x_{i}(t)-x_{c}(t)|^2}, \qquad
\sigma_{v}(t):=\sqrt{\frac{1}{N}\sum
\limits_{i}|v_{i}(t)-v_{c}(t)|^2},
\end{align*}
and use them to study the flocking behavior in \eqref{NL-CS}.
Indeed, the flocking conditions in definition \eqref{flock} are
equivalent to showing $\sup \limits_{t>0} \sigma_{x}(t)<\infty$ and
$\lim \limits_{t \to \infty}\sigma_{v}(t)=0$. The main flocking
result for the NL CS system was given in \cite{HaHaKi}.
\begin{proposition} We assume that assumptions $(A1)$-$(A2)$ hold and that
$(x(t),v(t))$ is a smooth solution to the NL CS system \eqref{NL-CS}
with the following constraint on initial configurations
\begin{equation} \label{flock-cond} \sigma_{v}^{3-2\gamma}(0) \leq C_{2} (3-2\gamma)
\int_{\sigma_{x}(0)}^{\infty}\psi(2\sqrt{N}s)\, ds , \end{equation}
for some $C_{2}=C_{2}(\gamma, C_{1}, N)>0$. Then, for $\gamma \in
(\frac{1}{2},\frac{3}{4})$ the particle system emerges to a flock.
\end{proposition}

The proof of Proposition 1 is actually pretty straightforward and a
brief sketch of it is possible in few lines. First, one easily shows
that the inequality for the dissipation of $\sigma_{v}(t)$ is
\begin{equation} \label{dis-sigma}\frac{d}{dt}\sigma_{v}(t) \leq
-C_{2}\psi(2\sqrt{N}\sigma_{x}(t)) \sigma_{v}^{2\gamma -1}
(t).\end{equation} Inequality \eqref{dis-sigma} would be enough to
ensure the convergence $\sigma_{v}(t)\to 0$, as long as we had a
uniform bound on $\sigma_{x}(t)$ (some $\sigma_{x}^{\infty}<\infty$
such that $\sup \limits_{t>0} \sigma_{x}(t)\leq
\sigma_{x}^{\infty}$). For this, we may use the Lyapunov functionals
\begin{equation*} \mathcal{E}^{\pm}(t):=\frac{\sigma_{v}^{3-2\gamma}(t)}{3-2\gamma}
\pm C_{2}\int_{0}^{\sigma_{x}(t)}\psi(2\sqrt{N}s)\,
ds,\end{equation*} and show (with the help of \eqref{dis-sigma})
that $\mathcal{E}^{\pm}(t)$ are dissipative
($\mathcal{E}^{\pm}(t)\leq \mathcal{E}^{\pm}(0)$ for $t\geq 0$).
Finally, using condition \eqref{flock-cond} and the dissipation of
$\mathcal{E}^{\pm}(t)$ it can be shown that $\sup \limits_{t>0}
\sigma_{x}(t)\leq \sigma_{x}^{\infty}<\infty$ which concludes with
the proof.

\begin{remark} It might appear that the flocking condition
\eqref{flock-cond} in Proposition 1 is an unnecessary restriction
but it is consistent with the phase transition character in the
classical CS model \eqref{CS}. This condition is satisfied trivially
in the case of long range interactions with $\int^{\infty}\psi(s)\,
ds = \infty$, giving unconditional flocking when the communication
weight $\psi(\cdot)$ has a heavy tail. When the interaction between
agents has a short range, the emergence of a flock is only
conditionally possible.
\end{remark}

We note that in the statement of Proposition 1, flocking is proven
for smooth solutions to system \eqref{NL-CS}. The well-posedeness of
system \eqref{NL-CS} is of course a separate problem and everything
depends on the regularity of interaction $\psi(\cdot)$ and coupling
$\Gamma(\cdot)$. Naturally, local well-posedeness can be proven for
more singular interactions between agents and velocity couplings
(see e.g. \cite{CaChHa}). For global results, the Lipshitz property
for both $\Gamma(\cdot)$ and $\psi(\cdot)$ is necessary. The
Lipshitz condition on $\Gamma(\cdot)$ for $\gamma \in
[1,\frac{3}{2})$ is a natural one, since the prototype example is
$\Gamma(v)=v|v|^{2(\gamma -1)}$, and hence this assumption is made
in our main result. On the other hand, for $\gamma \in
(\frac{1}{2},1)$ we may have non-uniqueness even for regular
communication weights. Our result gives a definite answer to the
existence of smooth solutions for $\gamma \in [1,\frac{3}{2})$ when
$\psi(s)=s^{-\alpha}$, $\alpha \geq 1$. For $\gamma \in
(\frac{1}{2},1)$ uniqueness is possible depending on the choice of
initial data and this problem remains open.

We now give the main result that we prove in the next section.
\begin{theorem} Consider the CS system \eqref{NL-CS} with
$\gamma \in (\frac{1}{2},\frac{3}{2})$ and initial data
$(x_{0},v_{0})$ that satisfy \begin{equation*} x_{i0}\neq
x_{j0}\qquad \text{for} \quad i \neq j .\end{equation*} We consider
the communication weight $\psi(s)=s^{-\alpha}$, with $\alpha \geq
1$. Furthermore, if $\gamma \in [1,\frac{3}{2})$ we assume that
$\Gamma(\cdot)$ is Lipshitz continuous. Then for any solution of the
NL CS system the particle trajectories remain non-collisional for
$t>0$.
\end{theorem}

The following easy lemma will prove helpful.
\begin{lemma} Given $p>0$, then for any $q>0$ there exists a constant
$C_{pq}:=C(p,q)>0$ such that \begin{equation*}|a^{-p}-b^{-p}|\geq
C_{pq}|a^q -b^q| \quad \text{for} \quad 0<a,b<1 .\end{equation*}
\end{lemma}

\begin{proof} This is an exercise in calculus. For $a=b$ it holds trivially.
If $a \neq b$, we set $x=a^q$, $y=b^q$ and we consider the function
$f(x,y)=\frac{y^{-p/q}-x^{-p/q}}{x-y}$ on the triangle $0<y<x<1$. We
can show that the function has a positive lower bound $C_{pq}>0$ for
any pair $p,q >0$.
\end{proof}

\section{Collision-avoiding for singular interactions.}

\begin{proof} The idea of the proof follows closely the steps in
\cite{CaChMuPe} in its first part. We assume that at some finite
time $t_{C}>0$ the first collision between a group of particle
happens. Then, based on this assumption and estimates that we derive
for the dynamics of the group of particles that collide, we reach a
contradiction. We denote the group of particles that collide at time
$t_{C}$ with $C$, and their number by $|C|$ i.e.
\begin{align*} &|x_{i}(t)-x_{j}(t)| \to 0 \quad \text{as}\quad t \nearrow t_{C}
\quad \text{for} \quad (i,j) \in C^2:=C \times C
\\ & |x_{i}(t)-x_{j}(t)|\geq \delta >0 \quad \text{for some} \quad \delta
>0, \quad (i,j)\not \in C^2, \quad t \in [0,t_{C}].
\end{align*}

We define the position and velocity fluctuation for the particles in
the collisional group by
\begin{equation*} \|x\|_{C}(t):=\sqrt{\sum \limits_{(i,j)
\in C^2}|x_{i}(t)-x_{j}(t)|^2} \quad \text{and} \quad
\|v\|_{C}(t):=\sqrt{\sum \limits_{(i,j) \in
C^2}|v_{i}(t)-v_{j}(t)|^2} .\end{equation*} Here $\sum
\limits_{(i,j)\in C^2}$ is the sum over all pairs $(i,j)$ where both
indices are members of group $C$. According to the definition we
just gave, we have that $\|x\|_{C}(t)\to 0$ as $t \nearrow t_{C}$.
We also have the following uniform bounds for $\|x\|_{C}(t)$ and
$\|v\|_{C}(t)$ as a result of the particle dynamics. There exist
$M>0$ and $R=R(t_{C})>0$, such that for all $t \in [0,t_{C}]$ we
have \begin{equation*} \|v\|_{C}(t)\leq M := \sqrt{2}|C| \, \sup
\limits_{i}|v_{i0}|, \quad \|x\|_{C}(t)\leq R := \sqrt{2}|C| \,(\sup
\limits_{i}|x_{i0}| +\sup \limits_{i}|v_{i0}|t_{C}).\end{equation*}
It is easy to show using the definition of $\|x\|_{C}(t)$ that
\begin{equation} \label{EqMot} \left| \frac{d}{dt}
\|x\|_{C}(t)\right| \leq \|v\|_{C}(t) . \end{equation} Our plan is
to show a sharp inequality for the dissipation of $\|v\|_{C}(t)$ in
the spirit of \cite{CaChMuPe}. In more detail, we show that:
\begin{itemize}
\item If
$\frac{1}{2}<\gamma <1$,
\begin{equation} \label{In1} \frac{d}{dt}\|v\|^{2}_{C}(t)
\leq -2c_{0}\psi(\|x\|_{C}(t))\|v\|_{C}^{2 \gamma}(t)
+2c_{1}\|x\|_{C}(t)\|v\|_{C}(t) +2c_{2}\|v\|_{C}(t) .\end{equation}
\item If $1\leq \gamma <\frac{3}{2}$,
\begin{equation} \label{In2} \frac{d}{dt}\|v\|^{2}_{C}(t)
\leq -2c_{0}\psi(\|x\|_{C}(t))\|v\|_{C}^{2 \gamma}(t)
+2c_{1}\|x\|_{C}(t)\|v\|_{C}(t) +2c_{2}\|v\|^{2}_{C}(t)
.\end{equation}
\end{itemize}
For the derivation of \eqref{In1}-\eqref{In2} we compute the time
evolution of $\|v\|_{C}(t)$, i.e.
\begin{align*} &\frac{d}{dt}\|v\|_{C}^2 =2\sum \limits_{(i,j)\in C^2} \left \langle
v_{i}-v_{j},\frac{1}{N}\sum
\limits_{k}\psi(|x_{k}-x_{i}|)\Gamma(v_{k}-v_{i})-\frac{1}{N} \sum
\limits_{k}\psi(|x_{k}-x_{j}|)\Gamma(v_{k}-v_{j})\right \rangle \\
\\ &=\frac{2}{N}\left( \sum_{\substack{(i,j)\in C^2 \\ k \in C}} +\sum_{\substack{(i,j)\in
C^2 \\ k \not \in C}} \right) \left( \psi(|x_{k}-x_{i}|)\left
\langle v_{i}-v_{j},\Gamma(v_{k}-v_{i}) \right \rangle  -
\psi(|x_{k}-x_{j}|)\left \langle
v_{i}-v_{j},\Gamma(v_{k}-v_{j})\right \rangle \right)
\\&=:J_{1}+J_{2} .
\end{align*}
The computation for the first term $J_{1}$ gives
\begin{align*} J_{1}&=\frac{2}{N} \sum_{\substack{(i,j)\in C^2 \\ k \in C}}
\left( \psi(|x_{k}-x_{i}|)\langle v_{i}-v_{j}, \Gamma(v_{k}-v_{i})
\rangle - \psi(|x_{k}-x_{j}|)\langle v_{i}-v_{j},
\Gamma(v_{k}-v_{j})\rangle \right)
\\ &\stackrel{i \leftrightarrow j}{=}\frac{4}{N} \sum_{\substack{(i,j)\in C^2 \\ k \in C}}
\psi(|x_{k}-x_{i}|)\langle v_{i}-v_{j}, \Gamma(v_{k}-v_{i}) \rangle
\\ &\stackrel{i \leftrightarrow k}{=}\frac{2}{N}\sum_{\substack{(i,j)\in C^2 \\ k \in C}}
\psi(|x_{k}-x_{i}|)\langle v_{i}-v_{j}, \Gamma(v_{k}-v_{i})\rangle +
\frac{2}{N}\sum_{\substack{(i,j)\in C^2 \\ k \in C}}
\psi(|x_{k}-x_{i}|)\langle v_{k}-v_{j}, \Gamma(v_{i}-v_{k})\rangle
\\ &=\frac{2}{N}\sum_{\substack{(i,j)\in C^2 \\ k \in C}}
\psi(|x_{k}-x_{i}|)\langle v_{i}-v_{k}, \Gamma(v_{k}-v_{i})\rangle
=-\frac{2}{N}\sum_{\substack{(i,j)\in C^2 \\ k \in C}}
\psi(|x_{k}-x_{i}|)\langle v_{k}-v_{i}, \Gamma(v_{k}-v_{i})\rangle
\\ &\stackrel{(A2)}{\leq}-\frac{2C_{1}|C|}{N} \sum \limits_{(i,j)\in C^2}
\psi(|x_{i}-x_{j}|)|v_{i}-v_{j}|^{2 \gamma} .
\end{align*}
Then, using the definition of $\|x\|_{C}$, $\|v\|_{C}$ and the
monotonicity of $\psi(\cdot)$
\begin{equation*} J_{1}\leq -2c_{0}\psi (\|x\|_{C})\|v\|_{C}^{2 \gamma}
\qquad \text{for}\quad c_{0}=\frac{C_{1}|C|}{N}.\end{equation*} For
$J_{2}$ we have
\begin{align*} J_{2}&=\frac{2}{N} \sum_{\substack{(i,j)\in C^2\\ k \not \in C}}
\psi(|x_{k}-x_{i}|)\langle v_{i}-v_{j}, \Gamma(v_{k}-v_{i}) \rangle
-\frac{2}{N}\sum_{\substack{(i,j)\in C^2 \\ k \not \in C}}
\psi(|x_{k}-x_{j}|)\langle v_{i}-v_{j}, \Gamma(v_{k}-v_{j}) \rangle
\\ &=\frac{2}{N}\sum_{\substack{(i,j)\in C^2 \\ k \not \in C}}
(\psi(|x_{k}-x_{i}|)-\psi(|x_{k}-x_{j}|))\langle v_{i}-v_{j},
\Gamma(v_{k}-v_{j}) \rangle
\\ &+\frac{2}{N}\sum_{\substack{(i,j)\in C^2
\\ k \not \in C}} \psi(|x_{k}-x_{i}|)\langle v_{i}-v_{j},
\Gamma(v_{k}-v_{i})-\Gamma(v_{k}-v_{j}) \rangle :=J_{21}+J_{22}
.\end{align*} The first term $J_{21}$ is bounded by
\begin{equation*} J_{21} \leq \frac{2}{N}
\Gamma_{M} L_{\delta}\sum_{\substack{(i,j)\in C^2
\\ k \not \in C}} |x_{i}-x_{j}|
|v_{i}-v_{j}|\leq 2 c_{1}\|x\|_{C}\|v \|_{C}, \quad
c_{1}=\frac{N-|C|}{N}\Gamma_{M}L_{\delta},\end{equation*} where
$L_{\delta}$ is the Lipshitz constant of $\psi(\cdot)$ on the
interval $(\delta, \infty)$ and $\Gamma_{M}:=\max \limits_{v}
|\Gamma(v)|$. Similarly, since $|x_{k}-x_{i}|>\delta$, it follows
that $\psi(|x_{k}-x_{i}|)<\psi(\delta)$ and we have the following
bounds for the second term $J_{22}$: \\ If $\frac{1}{2} <\gamma <1$,
\begin{equation*} J_{22} \leq \frac{4}{N} \psi(\delta) \Gamma_{M}\sum_{\substack{(i,j)\in C^2
\\ k \not \in C}} |v_{i}-v_{j}|\leq 2 c_{2}\|v\|_{C},\quad
c_{2}=\frac{2|C|(N-|C|)}{N}\psi(\delta)\Gamma_{M}, \end{equation*}
and if $1 \leq \gamma <\frac{3}{2}$ (using the Lipshitz property of
$\Gamma(\cdot)$, $|\Gamma(v)-\Gamma(w)|\leq L_{\Gamma}|v-w|$)
\begin{equation*} J_{22}\leq \frac{2}{N}\psi(\delta)
L_{\Gamma}\sum_{\substack{(i,j)\in C^2 \\ k \not \in C}}
|v_{i}-v_{j}|^2 \leq  2c_{2}\|v\|^{2}_{C}, \quad
c_{2}=\frac{N-|C|}{N}\psi(\delta) L_{\Gamma}. \end{equation*} The
derivation of estimates \eqref{In1}-\eqref{In2} is complete. We
mention a couple of differences with the linear case $\Gamma(v)=v$.
First, the term $J_{1}$ introduces the nonlinearity which makes it
impossible to use a differential Gronwall lemma given the additional
terms. Also, in contrast to the linear case where $J_{22}\leq 0$,
here $J_{22}$ is an extra term to be handled.

We keep in mind that for the singular weights $\psi(s)=s^{-\alpha}$
we are considering with $\alpha \geq 1$, their primitive
$\Psi(s)=\int^{s}\psi(t)dt$ is also singular at $0$. The following
bound on the increase of $\Psi(\|x\|_{C}(\cdot))$ on the interval
$(s,t)$ is useful.
\begin{align} \nonumber |\Psi(\|x \|_{C}(t))| &= \Big| \int_{s}^{t}
\frac{d}{d\tau}\Psi(\| x\|_{C}(\tau))\, d\tau + \Psi(\|x
\|_{C}(s))\Big|
\\ \nonumber &= \Big|\int_{s}^{t}\Psi'(\|x \|_{C}(\tau))
\frac{d}{d\tau}\|x \|_{C}(\tau)\, d\tau + \Psi(\|x \|_{C}(s))\Big|
\\ \label{PsiEst1} & \leq \int_{s}^{t}\psi(\|x \|_{C}(\tau))\|v \|_{C}(\tau)\,
d\tau + |\Psi(\|x \|_{C}(s))|. \end{align} If we can show that
$\Psi(\|x\|_{C}(t_{C}))<\infty$, the singularity of $\Psi(\cdot)$ at
$0$ implies that $\|x\|_{C}(t_{C})\neq 0$ which is a contradiction
to our initial hypothesis. In our study, we consider the cases
$\gamma \in (\frac{1}{2},1]$ and $\gamma \in (1,\frac{3}{2})$
separately. The first is rather trivial, but the latter requires a
bit of analysis.

$\bullet$ Case $\gamma \in (\frac{1}{2},1]$ : \\ The case of
$\frac{1}{2}<\gamma \leq 1$ is pretty straightforward. From estimate
\eqref{In1} we get directly that
\begin{equation*}\int_{s}^{t_{C}}
\psi(\|x\|_{C}(\tau))\|v\|^{2\gamma -1}_{C}(\tau)\, d\tau < \infty .
\end{equation*} We have that $0 \leq 2\gamma -1 \leq 1$ which,
combined with fact that $\|v\|_{C}(t)\leq M$, yields
\begin{equation*} \int_{s}^{t_{C}}
\psi(\|x\|_{C}(\tau))\|v\|_{C}(\tau)\, d\tau \leq
M^{2-2\gamma}\int_{s}^{t_{C}} \psi(\|x\|_{C}(\tau))\|v\|^{2\gamma
-1}_{C}(\tau)\, d\tau < \infty .
\end{equation*} In view of \eqref{PsiEst1} we have
$\Psi(\|x\|_{C}(t_{C}))<\infty$.

$\bullet$ Case $\gamma \in (1,\frac{3}{2})$ :
\\This case is more
elaborate. We know that $\|v\|_{C}(t)$ can only vanish at $t_{C}$,
otherwise because of \eqref{In2} it would be $0$ on some interval
$(s,t_{C})$ and $t_{C}$ cannot be the time of the first collision.
Thus, we have
\begin{equation}\label{In3} \frac{d}{dt}\|v\|_{C}(t)
\leq -c_{0}\psi(\|x\|_{C})\|v\|_{C}^{2 \gamma -1}(t)
+c_{1}\|x\|_{C}(t) +c_{2}\|v\|_{C}(t) .\end{equation} Although there
is no Gronwall lemma we can use for \eqref{In3}, we can reach a
contradiction doing a bit of qualitative analysis in \eqref{In3}.
The idea is actually pretty simple: For the three terms that appear
in the rhs of \eqref{In3} we study what happens when each of them is
the dominant as $t \nearrow t_{C}$ . For this, we consider the
following three cases
\begin{equation*}(C1) \quad \psi(\|x\|_{C}(t))\|v\|_{C}^{2 \gamma
-1}(t) <\|x\|_{C}(t),  \qquad (C2) \quad
\psi(\|x\|_{C}(t))\|v\|_{C}^{2 \gamma -1}(t)< \|v\|_{C}(t)
\end{equation*} and
\begin{equation*} (C3)\quad \frac{d}{dt}\|v\|_{C}(t)\leq
-\psi(\|x\|_{C}(t))\|v\|_{C}^{2\gamma -1}(t).
\end{equation*}
Notice that for now we make the assumption that constants
$c_{0}=c_{1}=c_{2}=1$. Later on we keep close track of all the
constants involved.

We begin by checking what happens when each of $(C1)$-$(C3)$ holds
on some interval $(t_{0},t_{C})$. When $(C1)$ holds on an interval
$(t_{0},t_{C})$, we show that $\|v\|_{C}$ is practically so small
that a collision cannot happen in finite time. Indeed, we have
\begin{equation*} \|v\|_{C}(t) < \left(
\frac{\|x\|_{C}(t)}{\psi(\|x\|_{C}(t))} \right)^{\frac{1}{2\gamma
-1}}=\|x\|_{C}^{\frac{\alpha +1}{2\gamma -1}}(t) \qquad \text{for}
\quad t \geq t_{0}.\end{equation*} By \eqref{EqMot}, we have
$\frac{d}{dt}\|x\|_{C}(t) > -\|v\|_{C}(t)\rightsquigarrow
\frac{d}{dt}\|x\|_{C}(t)>-\|x\|_{C}^{\frac{\alpha +1}{2\gamma
-1}}(t)$. We solve this differential inequality by integrating from
$t_{0}$ to $t$ to get
\begin{equation}\label{Sol1} \|x\|_{C}(t)
> \left(\|x\|_{C}^{\lambda}(t_{0})-\lambda
(t-t_{0})\right)^{\frac{1}{\lambda}},\end{equation} where $\lambda
=1-\frac{\alpha +1}{2\gamma -1}=\frac{2(\gamma -1)-\alpha}{2\gamma
-1}< 0$. Now, setting $t=t_{C}$ in \eqref{Sol1} leads to an obvious
contradiction since $\|x\|_{C}(t)>0$ for all $t\geq t_{0}$. It is
useful for our proof to compute the change of
$\Psi(\|x\|_{C}(\cdot))$ on the interval $(t_{0},t_{C})$.
\begin{align}\label{PsiEst2}|\Psi(\|x\|_{C}(t_{C}))|-|\Psi(\|x\|_{C}(t_{0}))|
\leq \int_{t_{0}}^{t_{C}}\psi(\|x\|_{C}(\tau)) \|v\|_{C}(\tau)\,
d\tau < \int_{t_{0}}^{t_{C}}\|x\|_{C}^{\mu}(\tau)\, d\tau
,\end{align} where $\mu=-\alpha +\frac{\alpha +1}{2\gamma
-1}=\frac{-2\alpha (\gamma -1)+1}{2\gamma -1}$. If $\mu \geq 0$, it
follows trivially that
$|\Psi(\|x\|_{C}(t_{C}))|-|\Psi(\|x\|_{C}(t_{0}))| \leq
R^{\mu}(t_{C}-t_{0})$. If on the other hand $\mu <0$, then we have
using \eqref{Sol1} in \eqref{PsiEst2}
\begin{align} \nonumber |\Psi(\|x\|_{C}(t_{C}))|-|\Psi(\|x\|_{C}(t_{0}))|
&< \int_{t_{0}}^{t_{C}} (\|x\|_{C}^{\lambda}(t_{0})-\lambda
(\tau-t_{0}))^{\nu} \, d\tau \\\nonumber & =\frac{1}{-\lambda (\nu
+1)}\left( \|x\|_{C}^{\lambda}(t_{0})-\lambda
(\tau-t_{0})\right)^{\nu +1}|_{t_{0}}^{t_{C}}
\\ \label{PsiEst3}&=-\frac{1}{\lambda (\nu +1)}
\left( (\|x\|_{C}^{\lambda}(t_{0})-\lambda (t_{C}-t_{0}))^{\nu
+1}-\|x\|_{C}^{\lambda(\nu +1)}(t_{0})\right) < \infty ,
\end{align} where $\nu=\frac{\mu}{\lambda}=\frac{-2\alpha (\gamma
-1)+1}{2(\gamma -1)-\alpha}>0$. The last inequality implies that
$|\Psi(\|x\|_{C}(t_{C}))|-|\Psi(\|x\|_{C}(t_{0}))| \leq
\frac{C_{\nu}}{\nu +1} (t_{C}-t_{0})$, for some constant
$C_{\nu}:=C(\nu)>0$, due to the basic inequality $|a^p-b^p| \leq
C_{p}|a-b|$, for $0<a,b<1$, $p\geq 1$.

Similarly, if $(C2)$ holds on some interval $(t_{0},t_{C})$ we have
\begin{equation*}  \|v\|_{C}(t)<\left(
\frac{1}{\psi(\|x\|_{C}(t))}\right)^{\frac{1}{2\gamma
-2}}=\|x\|_{C}^{\frac{\alpha}{2\gamma -2}}(t) \qquad \text{for}
\quad t \geq t_{0} .\end{equation*} By \eqref{EqMot}, we have
$\frac{d}{dt}\|x\|_{C}(t)>-\|x\|_{C}^{\frac{\alpha}{2(\gamma
-2)}}(t)$. The solution is once again given by \eqref{Sol1}, only
this time $\lambda=\frac{2(\gamma -1)-\alpha}{2(\gamma -1)}<0$.
Setting $t=t_{C}$ in \eqref{Sol1} we get a contradiction like
before. Also, the change of $\Psi(\|x\|_{C}(\cdot))$ over
$(t_{0},t_{C})$ is bounded like in \eqref{PsiEst2}, with
$\mu=-\alpha +\frac{\alpha}{2\gamma -2}$. Now
 $\mu>0$ and we get the estimate $|\Psi(\|x\|_{C}(t_{C}))|-|\Psi(\|x\|_{C}(t_{0}))| \leq
 R^{\mu}(t_{C}-t_{0})$. Overall, we have shown that $|\Psi(\|x\|_{C}(\cdot))|$
is Lipshitz in time when (C1) or (C2) holds, with a constant
$C_{\mu}=R^{\mu}$ for $\mu \geq 0$, or $C_{\mu}=\frac{C_{\nu}}{\nu
+1}$ for $\mu <0$.

Finally, we check what happens if $(C3)$ holds on interval
$(t_{0},t_{C})$. In this case, we show that although $\|v\|_{C}(t)$
is not necessarily ``small'', it dissipates so quickly to $0$ that
the collision cannot happen in finite time. Using a Grownall
inequality for $(C3)$ we have
\begin{equation} \label{Sol2}\|v\|_{C}(t) \leq \left( \|v\|^{2-2\gamma}_{C}(t_{0})
+2(\gamma -1)\int_{t_{0}}^{t}\psi(\|x\|_{C}(\tau))\, d\tau
\right)^{1/(2-2\gamma)} \qquad \text{for} \quad t \geq t_{0}.
\end{equation}
Then substituting eq. \eqref{Sol2} in \eqref{PsiEst1} we have
\begin{align*}  |\Psi(\|x \|_{C}(t_{C}))| &\leq  \int_{t_{0}}^{t_{C}}\psi(\|x \|_{C}(t)
\left( \|v\|^{2-2\gamma}_{C}(t_{0}) +2(\gamma
-1)\int_{t_{0}}^{t}\psi(\|x\|_{C}(\tau))\, d\tau
\right)^{\frac{1}{(2-2\gamma)}} dt +|\Psi(\|x \|_{C}(t_{0}))|
\\ &= -\frac{1}{3-2\gamma}\left( \|v\|^{2-2\gamma}_{C}(t_{0})
+2(\gamma -1)\int_{t_{0}}^{t}\psi(\|x\|_{C}(\tau))\, d\tau
\right)^{\frac{3-2\gamma}{2-2\gamma}} \Big|_{t=t_{0}}^{t=t_{C}}
+|\Psi(\|x\|_{C}(t_{0}))|
\\ &\leq  \frac{1}{3-2\gamma}\|v\|^{3-2\gamma}_{C}(t_{0})
+|\Psi(\|x\|_{C}(t_{0}))|< \infty .\end{align*} We may get a similar
estimate for (C3) by integrating the inequality
$\frac{1}{(3-2\gamma)}\frac{d}{dt}\|v\|^{3-2\gamma}_{C}\leq
-\psi(\|x\|_{C})\|v\|_{C}$ (which we get if we multiply (C3) by
$\|v\|^{2-2\gamma}_{C}$). Hence,
\begin{equation} \label{PsiEst4}|\Psi(\|x\|_{C}(t))|-|\Psi(\|x\|_{C}(s))| \leq -
\frac{1}{(3-2\gamma)}
(\|v\|^{3-2\gamma}_{C}(t)-\|v\|^{3-2\gamma}_{C}(s)).\end{equation}

We now proceed to the last part of the proof. We have investigated
what happens when there is an interval $(t_{0},t_{C})$ over which
one of the terms in the rhs of \eqref{In3} is dominant over the
others. Of course, this is by no means the only possible scenario.
In reality, we have to exclude the possibility of infinite
``oscillations'' between cases (C1)-(C3) right before the collision.
Therefore, we divide the interval $(t_{0},t_{C})$ into two regions
depending on the dominant terms in \eqref{In3}, i.e.
\begin{align}
\label{I1} & \frac{1}{2}c_{0}\psi(\|x\|_{C}(t))\|v\|_{C}^{2\gamma
-1}(t)< c_{1}\|x\|_{C}(t), \quad t \in I_{1}=(t_{0},t_{1})\cup
\ldots \cup(t_{2n},t_{2n+1})\cup \ldots
\\ &
\label{I2} \frac{d}{dt}\|v\|_{C}(t)<-\frac{1}{2}c_{0}
\psi(\|x\|_{C}(t))\|v\|_{C}^{2\gamma-1}(t)+c_{2}\|v\|_{C}(t), \quad
t \in I_{2}=(t_{1},t_{2})\cup \ldots \cup(t_{2n+1},t_{2n+2})\cup
\ldots
\end{align}

We start with region $I_{1}$ which is practically case (C1) that we
studied earlier. When $t \in I_{1}$, we have $\|v\|_{C}(t) < c_{3}
\|x\|^{\frac{\alpha +1}{2\gamma -1}}_{C}(t)$,
$c_{3}=\left(\frac{2c_{1}}{c_{0}}\right)^{1/(2\gamma -1)}$. We have
already proved the Lipshitz property of $\Psi(\|x\|_{C}(\cdot))$ on
any interval in $I_{1}$. Region $I_{2}$ is a ``hybrid'' of cases
(C2) and (C3), and we want to know how $\Psi(\|x\|_{C}(\cdot))$
changes on some interval in $I_{2}$ . First we multiply eq.
\eqref{I2} by $\|v\|^{2-2\gamma}_{C}(t)$ to get
\begin{align} \label{In4} \frac{1}{3-2\gamma}\frac{d}{dt}\|v\|^{3-2\gamma}_{C}(t)<
-\frac{1}{2}c_{0}\psi(\|x\|_{C}(t))\|v\|_{C}(t)+c_{2}\|v\|^{3-2\gamma}_{C}(t).\end{align}
Now using the definition of $\Psi(\cdot)$ and integrating on the
interval $(t_{2k+1},t_{2k+2})$ we get an expression which is the
equivalent to \eqref{PsiEst4} for ineq. \eqref{I2}
\begin{align} \label{PsiEst5}|\Psi(\|x\|_{C}(t_{2k+2}))|-|\Psi(\|x\|_{C}(t_{2k+1}))| &\leq -
\frac{2}{(3-2\gamma)c_{0}}
(\|v\|^{3-2\gamma}_{C}(t_{2k+2})-\|v\|^{3-2\gamma}_{C}(t_{2k+1}))
\\ \nonumber &+\frac{2c_{2}}{c_{0}}\int_{t_{2k+1}}^{t_{2k+2}}\|v\|^{3-2\gamma}_{C}(\tau)\,
d\tau .
\end{align}

The second rhs term in \eqref{PsiEst5} is bounded by
$c_{4}(t_{2k+2}-t_{2k+1})$, with
$c_{4}=\frac{2c_{2}}{c_{0}}M^{3-2\gamma}$. Taking the sum of the
second term from $k=1$ to $n$ we have $\sum \limits_{k=0}^{n}
\frac{2c_{2}}{c_{0}}\int_{t_{2k+1}}^{t_{2k+2}}\|v\|^{3-2\gamma}_{C}(\tau)\,
d\tau \leq c_{4} m(I_{2})$, where $m(I_{2})$ is the Lebesgue measure
of $I_{2}$. For the first term, we have from ineq. \eqref{In4} after
we integrate from $t_{2k+1}$ to $t_{2k+2}$ that
$\|v\|^{3-2\gamma}_{C}(t_{2k+2})-\|v\|^{3-2\gamma}_{C}(t_{2k+1})\leq
(3-2\gamma)c_{2}M^{3-2\gamma}(t_{2k+2}-t_{2k+1})$ and
\begin{equation*} \sum
\limits_{k=0}^{n}(\|v\|^{3-2\gamma}_{C}(t_{2k+2})-\|v\|^{3-2\gamma}_{C}(t_{2k+1}))\leq
(3-2\gamma)c_{2}M^{3-2\gamma} m(I_{2}) , \quad \forall n\geq 0.
\end{equation*} Of course, this bound is not enough since we need a
bound of this term from below!

The trick here lies in the fact that $\|v\|_{C}(t)$ does not change
drastically on $I_{1}$. We have shown that
$|\|x\|^{\lambda}_{C}(t_{2k+1})-\|x\|^{\lambda}_{C}(t_{2k})|\leq
-\lambda c_{3}(t_{2k+1}-t_{2k})$, where $\lambda=1-\frac{\alpha
+1}{2\gamma -1}<0$. Using the fact that
$\|v\|_{C}(t_{k})=c_{3}\|x\|^{\frac{\alpha +1}{2\gamma
-1}}_{C}(t_{k})$ and Lemma 2 (with $p=\frac{-\lambda (2\gamma
-1)}{\alpha +1}$, and $q=3-2\gamma$), we get
\begin{equation*}|\|v\|^{3-2\gamma}_{C}(t_{2k+1})-\|v\|^{3-2\gamma}_{C}(t_{2k})|<
c_{5}(t_{2k+1}-t_{2k}),\qquad \text{for} \quad c_{5}=\frac{-\lambda
c_{3}^{\frac{\lambda (2\gamma -1)}{\alpha +1}+1}}{C_{pq}}.
\end{equation*}  Hence, taking the sum we get
\begin{equation}\label{In5}\sum
\limits_{k=0}^{n}|\|v\|^{3-2\gamma}_{C}(t_{2k+1})-\|v\|^{3-2\gamma}_{C}(t_{2k})|<
c_{5} m(I_{1}), \quad \forall n\geq 0 .\end{equation} Finally, using
\eqref{In5} we have
\begin{equation*} -\sum \limits_{k=0}^{n}(\|v\|^{3-2\gamma}_{C}
(t_{2k+2})-\|v\|^{3-2\gamma}_{C}(t_{2k+1}))\leq c_{5}
m(I_{1})+\|v\|^{3-2\gamma}_{C}(t_{0})-\|v\|^{3-2\gamma}_{C}(t_{2n+2}),
\quad \forall n\geq 0 .\end{equation*}

We now decompose $|\Psi(\|x\|_{C}(t_{n}))|$ in the following manner
\begin{align*} |\Psi(\|x\|_{C}(t_{2n+2}))|&=\sum \limits_{k=0}^{n} \left(
|\Psi(\|x\|_{C}(t_{2k+2}))| -|\Psi(\|x\|_{C}(t_{2k+1}))| \right)
\\&+ \sum \limits_{k=0}^{n} \left( |\Psi(\|x\|_{C}(t_{2k+1}))| -
|\Psi(\|x\|_{C}(t_{2k}))| \right)+
|\Psi(\|x\|_{C}(t_{0}))|.\end{align*} We have done all the work
required to bound the two sums and show that
$|\Psi(\|x\|_{C}(t_{2n+2}))|<\infty$. We mention that a
decomposition could also be performed for
$|\Psi(\|x\|_{C}(t_{2n+1}))|$ with terms that are treated in similar
manner. We have shown how we can bound the first term of this
decomposition when we treated the terms that appear in
\eqref{PsiEst5}. Indeed,
\begin{align*} \sum \limits_{k=0}^{n} \left( |\Psi(\|x\|_{C}(t_{2k+2}))| -
|\Psi(\|x\|_{C}(t_{2k+1}))| \right) \leq \frac{2 c_{5}
m(I_{1})+2\|v\|^{3-2\gamma}_{C}(t_{0})-2\|v\|^{3-2\gamma}_{C}(t_{2n+2})}{(3-2\gamma)c_{0}}
+c_{4} m(I_{2}).\end{align*} The second term of the decomposition
can be easily bounded due to the Lipshitz property that we showed
for (C1), so
\begin{align*} \sum \limits_{k=0}^{n} \left( |\Psi(\|x\|_{C}(t_{2k+1}))| -
|\Psi(\|x\|_{C}(t_{2k}))| \right) \leq c_{3} C_{\mu}\sum
\limits_{k=0}^{n}(t_{2k+1}-t_{2k}) \leq c_{3} C_{\mu}
m(I_{1})<\infty \qquad \forall n\geq 0.\end{align*}

Putting those two sums together, we have
\begin{align*} |\Psi(\|x\|_{C}(t_{C}))|\leq \limsup
|\Psi(\|x\|_{C}(t_{n}))|&< c_{3}C_{\mu} m(I_{1})+c_{4}m(I_{2}) \\&+
\frac{2 c_{5}
m(I_{1})+2\|v\|^{3-2\gamma}_{C}(t_{0})}{(3-2\gamma)c_{0}}
+|\Psi(\|x\|_{C}(t_{0}))| <\infty ,\end{align*} which contradicts
our hypothesis of a collision at time $t_{C}$.

\end{proof}

\section{Uniform estimates on the particle distance for the
communication weight $\psi_{\delta}(s)=(s-\delta)^{-\alpha}$, with $
\alpha \geq 2 \gamma$.}

In this section, we give estimates for the minimum inter-particle
distance in the case of weights of the type
$\psi_{\delta}(s)=(s-\delta)^{-\alpha}$ for some fixed $\delta \geq
0$. We introduce the distance function $\mathcal{L}^{\beta}(t)$ for
the particle system $(x_{i}(t),v_{i}(t))$, with
$|x_{i}-x_{j}|>\delta$ for $1 \leq i \neq j \leq N$.
\begin{align*} \mathcal{L}^{\beta}(t):=\frac{1}{N(N-1)}\sum
\limits_{i\neq j}(|x_{i}(t)-x_{j}(t)|-\delta)^{-\beta}\qquad
\text{with} \quad \beta >0 .\end{align*} The symbol $\sum \limits_{i
\neq j}$ is short for the sum over all pairs $i,j$ for which $i \neq
j$. For the special case $\beta=0$ we define
$\mathcal{L}^{0}(t):=\frac{1}{N(N-1)}\sum \limits_{i\neq j}\log
(|x_{i}(t)-x_{j}(t)|-\delta)$.

This function is chosen so that if $\mathcal{L}^{\beta}(t)<\infty$
on some interval $[0,T]$, then it follows that particles do not
collide and that $|x_{i}(t)-x_{j}(t)|>\delta $ for $i \neq j$ on
$[0,T]$, provided that $|x_{i0}-x_{j0}|>\delta $ for $i\neq j$. We
will in fact show that if $\mathcal{L}^{\beta}(0)<\infty$ and given
any $T>0$, we have that $\mathcal{L}^{\beta}(t)< O(T)$ for all $t
\in [0,T]$. Of course, the choice of $\beta$-distance we use depends
directly on $\alpha$ and $\gamma$. In the spirit of \cite{CaChMuPe},
we introduce the maximal collisionless life-span of a solution with
initial datum $x_{0}$, i.e.
\begin{equation*} T_{0}:=\sup \{s \geq 0:  \forall \, \text{solution}\, \,
(x(t),v(t))\, \text{to problem \eqref{NL-CS}, there are no
collisions on}\, \, [0,s) \}.\end{equation*} We then prove
\begin{theorem} Suppose that $\alpha \geq 2 \gamma$ and that the CS
system has initial data $(x_{i0},v_{i0})$ satisfying
$|x_{i0}-x_{j0}|>\delta$ for $1 \leq i \neq j \leq N$. Then, for any
global smooth solution $(x(t),v(t))$ to the NL CS particle system we
have $T_{0}=\infty$. Moreover, we have the following estimates for
$t \in [0,T_{0})$:
\\(i) For $\alpha =2 \gamma$  we have
\begin{align} \label{Es1} \mathcal{L}^{0}(t)+
\frac{1}{2C_{1} \gamma (N-1)}\sum \limits_{i}|v_{i}(t)|^2 \leq
\frac{2 \gamma -1}{2 \gamma}t+ \mathcal{L}^{0}(0)+ \frac{1}{2 C_{1}
\gamma (N-1)}\sum \limits_{i}|v_{i0}|^2 .
\end{align}
\\(ii) For $\alpha >2 \gamma$ we choose $\beta=\frac{\alpha}{2\gamma}
-1$ and have
\begin{align}\label{Es2}\mathcal{L}^{\beta}(t)+
\frac{\beta}{2C_{1} \gamma (N-1)}\sum \limits_{i}|v_{i}(t)|^2 \leq
\frac{(2 \gamma -1)\beta}{2 \gamma}t +\mathcal{L}^{\beta}(0)+
\frac{\beta}{2C_{1} \gamma (N-1)}\sum \limits_{i}|v_{i0}|^2 .
\end{align}
\end{theorem}

\begin{proof}
First, observe that for $\beta > 0$ we have
\begin{equation*} \frac{d}{dt}\mathcal{L}^{\beta}(t)=-\frac{\beta}{N(N-1)}\sum
\limits_{i\neq j}(|x_{i}-x_{j}|-\delta)^{-\beta-1} \left \langle
\frac{x_{i}-x_{j}}{|x_{i}-x_{j}|}, v_{i}-v_{j}\right \rangle
\end{equation*}
and
\begin{equation*}\frac{d}{dt}\mathcal{L}^{0}(t)=\frac{1}{N(N-1)}\sum
\limits_{i\neq j}(|x_{i}-x_{j}|-\delta)^{-1} \left \langle
\frac{x_{i}-x_{j}}{|x_{i}-x_{j}|}, v_{i}-v_{j}\right\rangle
\end{equation*}
\\(i) If $\alpha =2 \gamma$, then we choose the $\beta$-distance with
$\beta=\alpha / 2 \gamma -1=0$. We then have
\begin{align*} \frac{d}{dt}\mathcal{L}^{0}(t)&=\frac{1}{N(N-1)}\sum \limits_{i \neq j}
(|x_{i}-x_{j}|-\delta)^{-\frac{\alpha}{2 \gamma}} \left \langle
\frac{x_{i}-x_{j}}{|x_{i}-x_{j}|}, v_{i}-v_{j} \right \rangle
\\
& \leq \frac{1}{N(N-1)}\left( \frac{2 \gamma -1}{2 \gamma}\sum
\limits_{i \neq j} 1 + \frac{1}{2 \gamma}\sum \limits_{i \neq j}
(|x_{i}-x_{j}|-\delta)^{-\alpha}|v_{i}-v_{j}|^{2 \gamma} \right)
\\ &\leq \frac{2 \gamma -1}{2 \gamma}+ \frac{1}{2\gamma N(N-1)}
\sum \limits_{i \neq j} \psi_{\delta}(|x_{i}-x_{j}|)|v_{i}-v_{j}|^{2
\gamma}.
\end{align*} Here we used Young's inequality $ab\leq
\frac{a^p}{p}+\frac{b^q}{q}$, for
$a=(|x_{i}-x_{j}|-\delta)^{-\frac{\alpha}{2 \gamma}}|v_{i}-v_{j}|$,
$b=1$ and $p=2 \gamma$, $q=\frac{2 \gamma}{2 \gamma -1}$. Finally,
by using the second moment estimate in \eqref{mom-eq} we have
\begin{equation*} \frac{d}{dt}\left( \mathcal{L}^{0}(t)+
\frac{1}{2C_{1} \gamma (N-1)}\sum \limits_{i}|v_{i}(t)|^2\right)
\leq \frac{2 \gamma -1}{2\gamma}.\end{equation*} Integrating from
$0$ to $t$ we get our estimate.
\\(ii) If $\alpha >2 \gamma$, we choose $\beta=\frac{\alpha}{2
\gamma}-1$ once again. We similarly have
\begin{align*} \frac{d}{dt}\mathcal{L}^{\beta}(t)\leq
\frac{(2\gamma -1)\beta}{2 \gamma}+ \frac{\beta}{2\gamma N(N-1)}
\sum \limits_{i \neq j} \psi_{\delta}(|x_{i}-x_{j}|)|v_{i}-v_{j}|^{2
\gamma},\end{align*} that yields
\begin{equation*} \frac{d}{dt}\left( \mathcal{L}^{\beta}(t)+
\frac{\beta}{2C_{1} \gamma (N-1)}\sum \limits_{i}|v_{i}(t)|^2\right)
\leq \frac{(2 \gamma -1)\beta}{2\gamma},\end{equation*} which gives
our estimate.

\end{proof}

\begin{remark} We note that the estimates \eqref{Es1}-\eqref{Es2} we just gave
generalize the ones in \cite{CaChMuPe}, as they are valid for any
$\gamma >\frac{1}{2}$. Another interesting observation is that there
is no need to use Gronwall's lemma for their derivation. As a
result, the minimum inter-particle distance estimate has a growth in
time that is $O(t)$ instead of $O(e^{Ct})$ which improves the
derived estimates in \cite{CaChMuPe}.
\end{remark}

We may now give a slightly more general version of the uniform
estimates presented above. For this, we assume that the
communication weight is not necessarily of the type
$\psi(s)=s^{-\alpha}$, but has a primitive $\Psi(\cdot)$ that is
singular at $0$, and the rate at which $\Psi$ becomes singular at
$0$ is sufficiently fast. We introduce a $\beta$-distance related to
this $\Psi(\cdot)$ by
\begin{equation*} \mathcal{L}^{\beta}(t):=\frac{1}{N(N-1)}
\sum \limits_{i \neq j} |\Psi(|x_{i}(t)-x_{j}(t)|)|^{\beta} \qquad
\text{for}\quad \beta>0.\end{equation*} Similarly, we define
$\mathcal{L}^{0}(t):=\frac{1}{N(N-1)}\sum \limits_{i \neq j}\log
|\Psi(|x_{i}(t)-x_{j}(t)|)|$. Then, with a computation based on
elementary techniques like in the previous result, we show

\begin{theorem} Consider system \eqref{NL-CS} with $\gamma > \frac{1}{2}$ and initial data
$(x_{0},v_{0})$ that satisfy \begin{equation*} x_{i0}\neq
x_{j0}\qquad \text{for} \quad i \neq j .\end{equation*} We also
assume that the communication weight $\psi(\cdot)$ has a primitive
$\Psi(s)$ that is singular at $s=0$ and satisfies the inequality
\begin{equation} \label{As1} \Psi'(s) \leq C
|\Psi(s)|^{(1-\beta)2\gamma /(2\gamma -1)} \qquad \text{for some}
\quad C>0 \, \, \, \text{and} \, \, \, 1> \beta \geq 0.
\end{equation} We have that any solution to \eqref{NL-CS} remains non-collisional for
$t>0$. Moreover, we have for $\beta>0$
\begin{align}\label{Es3}\mathcal{L}^{\beta}(t)+
\frac{\beta}{2C_{1} \gamma (N-1)}\sum \limits_{i}|v_{i}(t)|^2 \leq C
\frac{(2 \gamma -1)\beta}{2 \gamma}t +\mathcal{L}^{\beta}(0)+
\frac{\beta}{2C_{1} \gamma (N-1)}\sum \limits_{i}|v_{i0}|^2 ,
\end{align}
and for $\beta=0$
\begin{align} \label{Es4} \mathcal{L}^{0}(t)+
\frac{1}{2C_{1} \gamma (N-1)}\sum \limits_{i}|v_{i}(t)|^2 \leq C
\frac{2 \gamma -1}{2 \gamma}t+ \mathcal{L}^{0}(0)+ \frac{1}{2C_{1}
\gamma (N-1)}\sum \limits_{i}|v_{i0}|^2 .
\end{align}
\end{theorem}

\begin{proof} First, let us calculate the time evolution of
$\mathcal{L}^{\beta}(t)$ for $\beta >0$
\begin{align*} \frac{d}{dt}\mathcal{L}^{\beta}(t)&=\frac{\beta}{N(N-1)} \sum \limits_{i \neq j}
\Psi'(|x_{i}-x_{j}|)|\Psi(|x_{i}-x_{j}|)|^{\beta
-1}\frac{\Psi(|x_{i}-x_{j}|)}{|\Psi(|x_{i}-x_{j}|)|}\left \langle
\frac{x_{i}-x_{j}}{|x_{i}-x_{j}|},v_{i}-v_{j} \right \rangle
\\ &\leq \frac{\beta(2 \gamma -1)}{2 \gamma N(N-1)} \sum \limits_{i \neq j}\Psi'(|x_{i}-x_{j}|)
|\Psi(|x_{i}-x_{j}|)|^{(\beta -1)2\gamma /(2 \gamma -1)} \\
& \qquad \qquad + \frac{\beta}{2 \gamma N(N-1)} \sum \limits_{i \neq
j}\Psi'(|x_{i}-x_{j}|) |v_{i}-v_{j}|^{2 \gamma}.
\end{align*}

Once again we made use of Young's inequality for $a=|v_{i}-v_{j}|$,
$b=|\Psi(|x_{i}-x_{j}|)|^{\beta -1}$ and $p=2 \gamma$, $q=\frac{2
\gamma}{2 \gamma -1}$.

Now using condition \eqref{As1} and the second moment estimate in
\eqref{mom-eq} we get
\begin{equation*} \frac{d}{dt}\left( \mathcal{L}^{\beta}(t)
+\frac{\beta}{2C_{1}\gamma (N-1)}\sum
\limits_{i}|v_{i}(t)|^2\right)\leq C \frac{\beta (2\gamma
-1)}{2\gamma}.\end{equation*} For $\beta=0$, we get
\begin{equation*} \frac{d}{dt}\left( \mathcal{L}^{0}(t)
+\frac{1}{2C_{1}\gamma (N-1)}\sum \limits_{i}|v_{i}(t)|^2\right)\leq
C \frac{2\gamma -1}{2\gamma}.\end{equation*}
\end{proof}

E-MAIL: ioamarkou@iacm.forth.gr

\end{document}